\documentclass{amsart}

\usepackage{amsmath,amssymb,pinlabel}

\def\co{\colon\thinspace}

\newcommand{\CP}{\mathbb{C}\mathrm{P}}

\newcommand{\rmd}{\mathrm{d}}

\newcommand{\rme}{\mathrm{e}}

\newcommand{\frakp}{\mathfrak{p}}

\newcommand{\R}{\mathbb{R}}
\newcommand{\dR}{\mathrm{dR}}

\newcommand{\bfv}{\mathbf{v}}

\newcommand{\ttx}{\mathtt{x}}

\newcommand{\tty}{\mathtt{y}}

\newcommand{\Z}{\mathbb{Z}}

\DeclareMathOperator{\Int}{\mathrm{Int}}

\newtheorem{thm}{Theorem}[section]
\newtheorem{lem}[thm]{Lemma}
\newtheorem{prop}[thm]{Proposition}
\newtheorem{cor}[thm]{Corollary}

\theoremstyle{definition}
\newtheorem{defn}[thm]{Definition}

\theoremstyle{remark}
\newtheorem{rem}[thm]{Remark}
\newtheorem*{ack}{Acknowledgement}

\begin{document}

\title{Contact structures on principal circle bundles}

\author{Fan Ding}
\address{Department of Mathematics, Peking University,
Beijing 100871, P.~R. China}
\email{dingfan@math.pku.edu.cn}
\author{Hansj\"org Geiges}
\address{Mathematisches Institut, Universit\"at zu K\"oln,
Weyertal 86--90, 50931 K\"oln, Germany}
\email{geiges@math.uni-koeln.de}

\date{}


\thanks{F.~D.\ partially supported by
grant no.\ 10631060 of the National Natural Science Foundation
of China and a DAAD -- K.~C.~Wong fellowship,
grant no.\ A/09/99005, at the Universit\"at zu K\"oln.}

\begin{abstract}
We describe a necessary and sufficient condition for a principal
circle bundle over an even-dimensional manifold to carry an invariant
contact structure. As a corollary it is shown that all circle bundles
over a given base manifold carry an invariant contact structure, only
provided the trivial bundle does. In particular, all
circle bundles over $4$-manifolds admit invariant contact structures.
We also discuss the Bourgeois construction of contact structures
on odd-dimensional tori in this context, and we relate our results
to recent work of Massot, Niederkr\"uger and Wendl
on weak symplectic fillings in higher dimensions.
\end{abstract}

\maketitle

\section{Introduction}
The study of invariant contact structures on principal $S^1$-bundles
can be traced back to the work of Boothby and Wang~\cite{bowa58},
cf.~\cite[Section 7.2]{geig08}. They observed that
if the Euler class of the bundle can be represented by
a symplectic form on the base, a suitable connection $1$-form
will be an invariant contact form. Their main result was that any contact
form whose Reeb vector field is regular essentially arises in
this way. These contact structures are transverse to the $S^1$-fibres.

The classification of invariant contact structures on
$S^1$-bundles over surfaces was obtained by Lutz~\cite{lutz77}.
In the present paper we extend his results to higher
dimensions. We derive a necessary
(Proposition~\ref{prop:decomposition}) and sufficient
(Theorem~\ref{thm:existence}) condition for an $S^1$-bundle
to admit an invariant contact structure. Our explicit way of building
an invariant contact structure from a certain symplectic splitting
of the base manifold is inspired by the work of Giroux~\cite{giro91}
on convex hypersurfaces (corresponding to $\R$-invariant contact
structures), and it extends a construction by Stipsicz and
the second author~\cite{gest10} from trivial to nontrivial
bundles. In fact, as a consequence of our existence criterion we
can show that if the trivial $S^1$-bundle over a given base
manifold admits an invariant contact structure, then the same
is true for all nontrivial $S^1$-bundles (Corollary~\ref{cor:bundles}).
Combining this with the main result from~\cite{gest10}
we conclude that all $S^1$-bundles over $4$-manifolds carry
invariant contact structures.

In Section~\ref{subsection:Bourgeois} we discuss the relation of the
results in the present paper with Bourgeois's
construction of $T^2$-invariant contact structures
on $N\times T^2$, where $N$ is any closed contact manifold,
starting from an open book decomposition of~$N$.

In Section~\ref{subsection:twist} we demonstrate by an
example that the topology of
the total space of the $S^1$-bundle is not, in general,
determined by the symplectic splitting of the base.
However, as we prove in Section~\ref{section:Liouville},
once the topology of the bundle is fixed, the symplectic
splitting of the base determines the invariant contact structure
on the $S^1$-bundle up to equivariant diffeomorphism. This
classification result is formulated in terms of ideal
Liouville domains, a notion introduced by Giroux
and discussed at length by Massot et al.~\cite{mnw}.
\section{Conventions}
\label{section:conventions}
This section merely serves to fix our normalisation conventions
regarding principal connections.
Let $\pi\co M\rightarrow B$ be a principal $S^1$-bundle.
Write $\partial_{\theta}$ for the vector field on $M$
generating the $S^1$-action. By a connection $1$-form $\psi$ we mean an
$S^1$-invariant form on $M$, i.e.\ $L_{\partial_{\theta}}\psi\equiv 0$,
normalised by $\psi(\partial_{\theta})\equiv 1$.
Up to a factor $2\pi$ this $\psi$ is what Bott--Tu~\cite{botu82}
call the global angular form.

The $2$-form $\rmd\psi$ is then $S^1$-invariant and horizontal, where
the latter means that $i_{\partial_{\theta}}\rmd\psi\equiv 0$. It follows
that $\rmd\psi$ induces a closed $2$-form $\omega$ on~$B$, that is,
$\rmd\psi=\pi^*\omega$. This $2$-form $\omega$ is called the curvature form of
the connection~$\psi$. The (real) Euler class of the $S^1$-bundle is the
cohomology class given by $e=-[\omega/2\pi]\in H^2_{\dR}(B)$,
where $H^*_{\dR}$ denotes de~Rham cohomology. This Euler class is an
integral cohomology class in the sense that it lives in the image of the
inclusion $H^2(B;\Z)\subset H^2_{\dR}(B)$.

Given a further closed $2$-form $\omega'$ on $B$ with $[\omega']=[\omega]\in
H^2_{\dR}(B)$, one can find a connection $1$-form $\psi'$ with
$\rmd\psi'=\pi^*\omega'$. Indeed, we have $\omega'=\omega+\rmd\gamma$ for
some $1$-form $\gamma$ on~$B$, and we may then set $\psi'=\psi+\pi^*\gamma$.
The difference between two connection $1$-forms with the same
curvature form is a closed horizontal $1$-form.

By slight abuse of notation we shall not usually distinguish
between an $S^1$-invariant and horizontal differential form
on $M$ and the induced form on~$B$.

Principal $S^1$-bundles over $B$ are classified by their
Euler class in integral cohomology $H^2(B;\Z)$. The
real Euler class $e\in H^2_{\dR}(B)$ does not
completely classify $S^1$-bundles in the presence of torsion
in $H^2(B;\Z)$; there are nontrivial $S^1$-bundles with flat
connections, see~\cite{koba56}. This issue will be of relevance
in Section~\ref{section:Liouville} only.
\section{Symplectic decompositions}
\label{section:decomposition}
Assume now that $B$ is a closed, connected, oriented $2n$-manifold
and $\pi\co M\rightarrow B$ a principal $S^1$-bundle as
before, with the corresponding orientation on~$M$. Suppose that $M$ admits
an $S^1$-invariant cooriented contact structure $\xi=\ker\alpha$ such that
$\alpha\wedge (\rmd\alpha)^n$ is a (positive) volume form
for this orientation of~$M$. The contact form $\alpha$ may likewise be taken
to be $S^1$-invariant, for we can always pass to the averaged
contact form
\[ \int_{\theta\in S^1}\theta^*\alpha.\]

Then $u:=\alpha(\partial_{\theta})$ defines a smooth $S^1$-invariant
function on~$M$.
With $\psi$ a connection $1$-form on~$M$,
define a $1$-form $\beta$ on $M$ by
\[ \alpha =\beta +u\psi.\]
This $\beta$ is $S^1$-invariant and horizontal. Thus, both the
function $u$ and the $1$-form $\beta$ descend to~$B$.

Write $\omega$ for the curvature form of $\psi$ as in
the previous section.

\begin{lem}
\label{lem:volume}
The $2n$-form
\[ \Omega:= (\rmd\beta +u\omega)^{n-1}\wedge
   [n\beta\wedge\rmd u+u(\rmd\beta+u\omega)] \]
is a volume form on~$B$.
\end{lem}

\begin{proof}
A straightforward computation gives $\alpha\wedge (\rmd\alpha)^n=
\psi\wedge\Omega$.
\end{proof}

The terminology in the following definition is chosen because of the
obvious analogy with the theory of $\R$-invariant contact structures
near convex hypersurfaces in the sense of Giroux~\cite{giro91},
cf.~\cite[Definition~4.3]{giro01}.

\begin{defn}
The {\em dividing set\/} of $B$ induced by the contact structure $\xi$
is the set
\[ \Gamma:=\{ p\in B\co u(p)=0\} .\]
We write
\[ B_{\pm}:=\{ p\in B\co \pm u(p)\geq 0\}, \]
so that $B\setminus\Gamma=\Int(B_+)\sqcup\Int(B_-)$.
\end{defn}

In \cite{lutz77} it was shown that for $\dim B=2$
this dividing set classifies invariant contact structures
up to equivariant diffeomorphism. The importance of this dividing set
even for the classification of non-invariant contact structures
on circle bundles over surfaces was observed by Giroux~\cite{giro01}. 

\begin{lem}
\label{lem:dividing}
The dividing set\/ $\Gamma$ is a (possibly empty)
closed codimension~$1$ submanifold of~$B$.
The $1$-form $\beta_0:=\beta|_{T\Gamma}$ is a contact form on $\Gamma$,
inducing the orientation of\/ $\Gamma$ as the boundary of~$B_+$.
\end{lem}

\begin{proof}
From the preceding lemma it follows that
\[ -\rmd u\wedge\beta\wedge (\rmd\beta)^{n-1} >0\]
along~$\Gamma$, and $-\rmd u$ evaluates positively on vectors
pointing out of~$B_+$.
\end{proof}

We now want to show that there are symplectic forms on
$B_{\pm}$ representing the Euler class of the bundle
and compatible with the contact structure $\ker\beta_0$ on
the boundary in the sense of the following definition.

\begin{defn}
\label{defn:weak}
Let $(W,\omega)$ be a compact symplectic manifold of dimension~$2n$,
oriented by the symplectic form~$\omega$, and $\eta=\ker\beta_0$
a cooriented (and hence oriented) contact structure on $\partial W$
inducing the boundary orientation. We say that $(W,\omega)$ is a
{\em weak filling\/} of $(\partial W,\eta)$ if
\[ (\rmd\beta_0)^k\wedge\omega^{n-1-k}|_{\eta}>0\;\;\mbox{\rm for all}\;\;
k=0,\dots ,n-1.\tag{$\mathrm{w}_1$}\]
\end{defn}

Observe that this condition does not depend on the choice
of contact form for~$\eta$ (among forms inducing the given
coorientation). This definition of a weak filling is
essentially the one proposed by Massot et al.~\cite{mnw}
(in contrast with earlier definitions, where only $\omega^{n-1}|_{\eta}>0$
had been required). In fact, they demand
\[ (b\,\rmd\beta_0+\omega)^{n-1}|_{\eta}>0\;\;
\mbox{\rm for all smooth functions $b\co\partial W\rightarrow\R_0^+$,}
\tag{$\mathrm{w}_2$} \]
which on the face of it is weaker than $(\mathrm{w}_1)$.
Lemma~\ref{lem:collar} shows that by a modification in
a collar one can pass from $(\mathrm{w}_2)$ to
$(\mathrm{w}_1)$, see Remark~\ref{rem:weak}.
This observation can be phrased as follows.

\begin{prop}
\label{prop:weak}
Weak symplectic fillability in the sense of\/ $(\mathrm{w}_2)$
is equivalent to fillability in the sense of\/ $(\mathrm{w}_1)$.
\qed
\end{prop}

So for our purposes $(\mathrm{w}_1)$ and
$(\mathrm{w}_2)$ can be used interchangeably.
Beware, though, that a given symplectic filling may satisfy
$(\mathrm{w}_2)$, but not $(\mathrm{w}_1)$.

\begin{prop}
\label{prop:decomposition}
Given a principal $S^1$-bundle $\pi\co M\rightarrow B$ of Euler class
$e$ with an $S^1$-invariant contact structure~$\xi$, then with
notation as above the following holds.
There are symplectic forms $\omega_{\pm}$ on $\pm B_{\pm}$,
where $-B_-$ denotes $B_-$ with reversed orientation,
such that
\begin{enumerate}
\item[(i)] $\mp [\omega_{\pm}/2\pi]=e|_{B_{\pm}}$;
\item[(ii)] if\/ $\Gamma$ is non-empty, then $(\pm B_{\pm},\omega_{\pm})$
are weak fillings of\/ $(\Gamma,\ker\beta_0)$.
\end{enumerate}
\end{prop}

\begin{rem}
\label{rem:interior}
As we shall see presently,
an $S^1$-invariant contact form $\alpha$ on $M$
gives rise in a natural way to symplectic forms $\omega_{\pm}$ on
the \emph{interior} of $\pm B_{\pm}$. The symplectic forms
in the statement of the proposition are obtained by considering
a diffeomorphic copy of $\pm B_{\pm}$ obtained by shrinking
a collar neighbourhood of its boundary. Conversely,
in the proof of Theorem~\ref{thm:existence}, where
we construct an $S^1$-invariant contact form on $M$ from
symplectic forms on $\pm B_{\pm}$, we actually
insert a `neck' $[-1,1]\times\Gamma$ between $B_+$ and $B_-$
and find a contact structure on the corresponding $S^1$-bundle over
this enlarged base manifold. The approach 
in Section~\ref{section:Liouville} via so-called ideal Liouville
domains circumvents these subtleties.
\end{rem}

\begin{proof}[Proof of Proposition~\ref{prop:decomposition}]
On $B\setminus\Gamma$ one can write the volume form $\Omega$
from Lemma~\ref{lem:volume} as
\[ \Omega =u^{n+1}\bigl(\rmd (\beta/u)+\omega\bigr)^n,\]
so
\[ \omega_{\pm}:= \pm\bigl( \rmd (\beta/u)+\omega\bigr)|_{\Int(B_{\pm})} \]
are symplectic forms on $\Int(B_{\pm})$, respectively,
inducing the positive orientation on $\Int(B_+)$
and the negative orientation on~$\Int(B_-)$.
Moreover, we have
\[ \mp [\omega_{\pm}/2\pi]=-[\omega/2\pi]|_{\Int(B_{\pm})}
=e|_{\Int(B_{\pm})}.\]

If $\Gamma\neq\emptyset$, we may choose $\varepsilon>0$ so small that
\[ B_{\pm}^{\varepsilon}:=\{ p\in B\co \pm u(p)\geq\varepsilon\} \]
is an isotopic copy of $B_{\pm}$ in~$B$, and such that
for each $s\in[-\varepsilon,\varepsilon]$ the set
\[ \Gamma_s:=\{ p\in B\co u(p)=s\} \]
is an isotopic copy of $\Gamma$ in $B$ with $\beta_s:=\beta|_{T\Gamma_s}$
a contact form. In other words, $B_{\pm}^{\varepsilon}$ is obtained
from $B_{\pm}$ by shrinking a collar neighbourhood of its
boundary. By Gray stability \cite[Theorem~2.2.2]{geig08},
the contact manifolds $(\Gamma,\ker\beta_0)$
and $(\Gamma_{\pm\varepsilon},\ker\beta_{\pm\varepsilon})$ are diffeomorphic.

On $\Gamma_{\varepsilon}$ we have for all $k=0,\dots ,n-1$,
possibly after choosing a smaller $\varepsilon>0$,
\[ \beta_{\varepsilon}\wedge (\rmd\beta_{\varepsilon})^k\wedge
\omega_+^{n-1-k}|_{T\Gamma_{\varepsilon}}=
\beta\wedge (\rmd\beta)^k\wedge
(\rmd\beta/\varepsilon +\omega)^{n-1-k}|_{T\Gamma_{\varepsilon}} >0,\]
so $(B_+^{\varepsilon},\omega_+|_{B_+^{\varepsilon}})$
is a weak filling of $(\Gamma_{\varepsilon},\ker\beta_{\varepsilon})
\cong (\Gamma,\ker\beta_0)$. Condition (i) holds under the obvious
diffeomorphism between $B_+^{\varepsilon}$ and~$B_+$.

Similarly, on $\Gamma_{-\varepsilon}$ we have
\[ \beta_{-\varepsilon}\wedge (\rmd\beta_{-\varepsilon})^k\wedge
\omega_-^{n-1-k}|_{T\Gamma_{-\varepsilon}}=
\beta\wedge (\rmd\beta)^k\wedge
(\rmd\beta/\varepsilon -\omega)^{n-1-k}|_{T\Gamma_{-\varepsilon}} >0,\]
so $(-B_-^{\varepsilon},\omega_-|_{B_-^{\varepsilon}})$
is a weak filling of $(\Gamma_{-\varepsilon},\ker\beta_{-\varepsilon})
\cong (\Gamma,\ker\beta_0)$ in the sense of~$(\mathrm{w}_1)$.
\end{proof}
\section{Constructing an invariant contact structure}
We are now going to show that the conditions listed
in Proposition~\ref{prop:decomposition} are in fact also sufficient
for the existence of an $S^1$-invariant contact structure on~$M$.

\begin{thm}
\label{thm:existence}
Let $\pi\co M\rightarrow B$ be a principal $S^1$-bundle of
Euler class~$e$ over a closed, connected, oriented
manifold $B$ of dimension~$2n$. Suppose that $B$ admits a splitting
$B=\nobreak{B_+\cup_{\Gamma}B_-}$ along a (possibly empty)
codimension~$1$ submanifold\/ $\Gamma$ such that there are
symplectic forms $\omega_{\pm}$ on $\pm B_{\pm}$ and
a cooriented contact structure $\ker\beta$ on $\Gamma$
satisfying conditions {\rm (i)} and {\rm (ii)}
of Proposition~\ref{prop:decomposition}. Then $M$ admits
an $S^1$-invariant contact structure with dividing set~$\Gamma$.
\end{thm}

\begin{proof}[Proof of Theorem \ref{thm:existence} for $\Gamma=\emptyset$]
If $\omega_+$ is a symplectic form on $B$ with $-[\omega_+/2\pi]=e$,
take $\alpha$ to be a connection $1$-form $\psi$
with curvature form~$\omega_+$.
Then we have $\alpha\wedge (\rmd\alpha )^n=\psi\wedge\pi^*\omega_+^n>0$.
If $-B$ admits a symplectic form $\omega_-$ with $[\omega_-/2\pi]=e$,
set $\alpha=-\psi$, where $\psi$ is a connection $1$-form with
curvature form $-\omega_-$. Then $\alpha\wedge (\rmd\alpha )^n=
-\psi\wedge\pi^*\omega_-^n>0$.
\end{proof}

From now on it will be assumed that $B$ decomposes as
$B=B_+\cup_{\Gamma}B_-$ with $\Gamma\neq\emptyset$.
We begin by considering $B_+$ and $B_-$ separately. Our first aim
is to modify $\omega_{\pm}$ in a neighbourhood of the boundary
such that the new symplectic manifolds resemble strong fillings
of $(\Gamma,\ker\beta)$. The next lemma mildly generalises an
idea of Eliashberg~\cite{elia91}, cf.~\cite{geig06}.

\begin{lem}
\label{lem:collar}
The symplectic forms $\omega_{\pm}$ can be modified in
a collar neighbourhood of $\Gamma=\partial (\pm B_{\pm})$
in $\pm B_{\pm}$ such that in a smaller collar
neighbourhood $(-\varepsilon ,0]\times\Gamma$ we can write
\[ \omega_{\pm}=\pm \omega_{\pm}^{\Gamma}+\rmd (\rme^s\beta),\]
possibly after replacing $\beta$ by $K\beta$ for some
large $K\in\R^+$. Here $\omega_{\pm}^{\Gamma}$ are
$2$-forms on~$\Gamma$, pulled back to $(-\varepsilon ,0]\times\Gamma$
under the projection map to~$\Gamma$.
\end{lem}

\begin{rem}
\label{rem:weak}
For the proof we shall only assume that $(\pm B_{\pm},\omega_{\pm})$
are weak fillings in the sense of~$(\mathrm{w}_2)$.
After the modification described in the lemma, and possibly
after extending the new $\omega_{\pm}$ in the obvious way
over an attached collar $[0,R]\times\Gamma$ for some large $R>0$,
we obviously have a weak filling in the sense of~$(\mathrm{w}_1)$.
This proves Proposition~\ref{prop:weak}.
\end{rem}

\begin{proof}[Proof of Lemma~\ref{lem:collar}]
For ease of notation we first consider $(B_+,\omega_+)$.
Consider a tubular neighbourhood
$[0,1]\times\Gamma$ of the boundary, where $\{ 1\}\times\Gamma\equiv
\Gamma=\partial B_+$. Define the $2$-form $\omega_+^{\Gamma}$
on $[0,1]\times\Gamma$ by first restricting
$\omega_+$ to $T(\{ 0\}\times\Gamma)$ (i.e.\
pulling back under the inclusion $\{ 0\}\times\Gamma\subset [0,1]
\times\Gamma$) and then pulling back again to $[0,1]\times\Gamma$.
Then both forms $\omega_+$ and $\omega_+^{\Gamma}$ represent the
cohomology class $-2\pi e|_{[0,1]\times\Gamma}$, so there is a $1$-form
$\gamma$ on $[0,1]\times\Gamma$ such that
\[ \omega_+=\omega_+^{\Gamma}+\rmd\gamma.\]

We continue to write $\beta$ for the $1$-form on
$[0,1]\times\Gamma$ obtained by pulling back the original $\beta$
from $\Gamma=\{ 1\}\times\Gamma$.
Since $(B_+,\omega_+)$ is a weak filling of $(\Gamma,\ker\beta)$,
we may assume that the collar $[0,1]\times\Gamma$ had been chosen
so small that condition $(\mathrm{w}_2)$ is satisfied on
the tangent bundle $T(\{ t\}\times\Gamma)$ for each $t\in [0,1]$,
and for any convex linear combination of $\omega_+^{\Gamma}$
and $\omega_+$. In other words, for any $c_t\in [0,1]$
and function $b_t\co \{ t\}\times\Gamma\rightarrow\R_0^+$
we have
\[ \beta\wedge (b_t\,\rmd\beta +\omega_+^{\Gamma}+
c_t\,\rmd\gamma)^{n-1} >0\;\;\mbox{\rm on}\;\; T(\{ t\}\times\Gamma).\]

In the sequel we shall take $b_t$ to be constant
on $\{ t\}\times\Gamma$, so we regard both $t\mapsto b_t$
and $t\mapsto c_t$ as real-valued functions on $[0,1]$.
Set
\[ \tilde{\omega}_+=\omega_+^{\Gamma}+\rmd (c\gamma)+\rmd (b\beta)\]
on $[0,1]\times\Gamma$, where the smooth functions
$b$ and $c$ on $[0,1]$, are
chosen as follows. Fix a small $\varepsilon >0$.
Choose $b\co [0,1]\rightarrow\R^+_0$
monotonically increasing, identically $0$ near $t=0$ and with $b'(t)>0$
for $t>\varepsilon /2$.
Choose $c\co
[0,1]\rightarrow [0,1]$ identically $1$ on $[0,\varepsilon ]$ and
identically $0$ near~$t=1$.

We compute
\[ \tilde{\omega}_+^n=\bigl( n\,\rmd t\wedge (b'\beta+c'\gamma)+
(b\,\rmd\beta+\omega_+^{\Gamma}+c\,\rmd\gamma)\bigr)\wedge
(b\,\rmd\beta+\omega_+^{\Gamma}+c\,\rmd\gamma)^{n-1}.\]

By our choices, the $2n$-form
\[ \rmd t\wedge\beta\wedge(b\,\rmd\beta+\omega_+^{\Gamma}+
c\,\rmd\gamma)^{n-1} \]
is a volume form on $[0,1]\times\Gamma$, so is
\[ (\omega_+^{\Gamma}+c\,\rmd\gamma)^n\]
on $[0,\varepsilon]\times\Gamma$,
where $c=1$, since $\omega_+^{\Gamma}+\rmd\gamma=\omega_+$.
Thus, by choosing $b$ small on $[0,\varepsilon]$ and
$b'$ large compared with $\max\{1,b,|c'|\}$
on $[\varepsilon ,1]$ (up to some multiplicative constants
involving the norms of the differential forms in the
above expression), one can ensure that $\tilde{\omega}_+^n>0$.
Then $\tilde{\omega}_+$ is a symplectic form on $[0,1]\times\Gamma$
and, in terms of the coordinate $s:=\log b(t)-\log b(1)$, this symplectic
form looks like $\omega_+^{\Gamma}+\rmd (\rme^sb(1)\beta)$ near
$\{1\}\times\Gamma$.

For $(-B_-,\omega_-)$ the argument is completely analogous,
except that we take $\omega_-^{\Gamma}$ to be the
restriction of $-\omega_-$ to $T(\{ 0\}\times\Gamma)$.
The value $b(1)$ may be chosen the same for $\omega_+$
and~$\omega_-$.
\end{proof}

\begin{rem}
Our choice of sign in the preceding lemma implies that
when we regard the $2$-forms $\omega_{\pm}^{\Gamma}$
as forms on~$\Gamma$, we have $-[\omega_{\pm}^{\Gamma}/2\pi]
=e|_{\Gamma}$, so both forms $\omega_{\pm}^{\Gamma}$
are curvature forms for the restriction of
the $S^1$-bundle to~$\Gamma$.
\end{rem}

The following is a generalisation of the argument used for
proving \cite[Theorem~1]{gest10}.

\begin{proof}[Proof of Theorem~\ref{thm:existence} for $\Gamma\neq\emptyset$]
By Lemma~\ref{lem:collar} we find a collar neighbourhood
\[ (-1-\varepsilon,-1]\times\Gamma \]
of $\{-1\}\times\Gamma\equiv\Gamma=
\partial B_+$ in $B_+$ where
\[ \omega_+=\omega_+^{\Gamma}+\rmd (\rme^{t+1}\beta);\]
the shift in the collar parameter is made for notational convenience below.
Likewise, we have a collar neighbourhood
\[ [1,1+\varepsilon)\times\Gamma \]
of $\{1\}\times\Gamma\equiv\Gamma=
\partial (-B_-)$ in $B_-$ where
\[ \omega_-=-\omega_-^{\Gamma}+\rmd (\rme^{-t+1}\beta);\]

Write the base $B$ of the $S^1$-bundle as
\[ B_+\cup_\Gamma([-1,1]\times\Gamma)\cup_{\Gamma}B_-.\]
Let $\psi_{\pm}$ be connection $1$-forms of the restriction of
the $S^1$-bundle to $B_{\pm}$ with curvature forms
$\pm\omega_{\pm}$. Let $\psi_{\pm}^{\Gamma}$ be
connection $1$-forms of the $S^1$-bundle over $\Gamma$
with curvature form $\omega_{\pm}^{\Gamma}$.

\begin{lem}
The choices in the preceding argument can be made in such a way that over
the two collars $(-1-\varepsilon,-1]\times\Gamma$
and $[1,1+\varepsilon)\times\Gamma$ we have
\[ \psi_{\pm}=\psi_{\pm}^{\Gamma}\pm\rme^{\pm t+1}\beta,\]
respectively, perhaps at the cost of taking a slightly
smaller $\varepsilon>0$.
\end{lem}

\begin{proof}
We only deal with $\psi_+$; the argument for $\psi_-$ is completely
analogous. Over the collar $(-1-\varepsilon,-1]\times\Gamma$ the
connection forms $\psi_+$ and $\psi_+^{\Gamma}+\rme^{t+1}\beta$ have the same
curvature form~$\omega_+$. It follows that
\[ \psi_+=\psi_+^{\Gamma} +\rme^{t+1}\beta+\gamma \]
with $\gamma$ a closed horizontal $1$-form.
Choose a closed $1$-form $\gamma^{\Gamma}$ on~$\Gamma$, which we also
interpret as a $1$-form on $(-1-\varepsilon,-1]\times\Gamma$,
representing the same class as $\gamma$ in $H^1_{\dR}
((-1-\varepsilon,-1]\times\Gamma)$. Then we can write
\[ \gamma =\gamma^{\Gamma}+\rmd h \]
for some smooth function on $(-1-\varepsilon,-1]\times\Gamma$.
Replace $\psi_+^{\Gamma}$ by $\psi_+^{\Gamma}+\gamma^{\Gamma}$, and
$\psi_+$ by $\psi_+ -\rmd (\chi h)$, where
$\chi\co (-1-\varepsilon,-1]\rightarrow [0,1]$
interpolates smoothly between $0$ near $-1-\varepsilon$ and $1$ near~$-1$.
Then the new $\psi_+$ still extends as before over~$B_+$, and near
$\{ -1\}\times\Gamma$ we have the equality claimed in the lemma.
\end{proof}

We continue with the proof of Theorem~\ref{thm:existence}.
Let $\psi_t^{\Gamma}$, $t\in [-1,1]$, be a smooth family
of connection $1$-forms on the $S^1$-bundle over $\Gamma$ with
$\psi_t^{\Gamma}=\psi_{\pm}^{\Gamma}$ for $t$ near~$\mp 1$.

Now choose two smooth functions $f$ and $g$ on the interval 
$(-1-\varepsilon , 1+\varepsilon )$ subject to the following conditions
(see Figure~\ref{figure:functions}):
\begin{itemize} 
\item $f$ is an even and nowhere zero function with
      $f(t)=\rme^{t+1}$ near $(-1-\varepsilon ,-1]$,
\item $g$ is an odd function with $g(t)=1$ near $(-1-\varepsilon ,-1]$
      and a single zero at~$0$,
\item $f'g-fg'>0$,
\item $f\gg 1$ and $f'g-fg'\gg 1$ where $g'\neq 0$.
\end{itemize}

\begin{figure}[h]
\labellist 
\small\hair 2pt
\pinlabel $t$ [t] at 246 145
\pinlabel $t$ [t] at 605 145
\pinlabel $g(t)$ [r] at 486 282
\pinlabel $f(t)$ [r] at 126 282
\pinlabel $1$ [r] at 117 217
\pinlabel $1$ [r] at 477 217
\pinlabel $-1$ [t] at 54 135
\pinlabel $1$ [t] at 198 135
\pinlabel $-1$ [t] at 414 135
\pinlabel $1$ [t] at 558 135
\endlabellist
\centering
\includegraphics[scale=0.5]{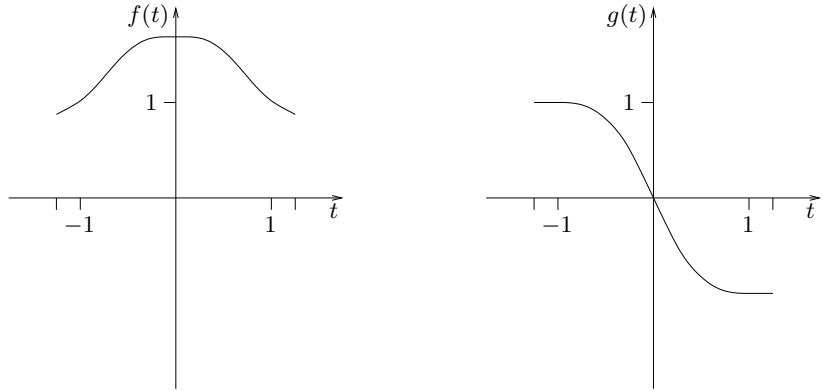}
  \caption{The functions $f$ and $g$.}
  \label{figure:functions}
\end{figure}

Define a smooth $S^1$-invariant $1$-form $\alpha$ on the $S^1$-bundle
over $B$ by
\[ \alpha=\begin{cases}
\psi_+                    & \text{over $B_+$},\\
f\beta+g\psi_t^{\Gamma} & \text{over $[-1,1]\times\Gamma$},\\
-\psi_-                   & \text{over $B_-$}.
\end{cases} \]
Over $B_{\pm}$ this defines a contact form by the same computation as
in the proof for the case $\Gamma=\emptyset$.
Over $[-1,1]\times\Gamma$ we compute
\begin{eqnarray*}
\alpha\wedge (\rmd\alpha)^n & = &
    (f\beta+g\psi_t^{\Gamma})\wedge
    n\, (f\, \rmd\beta+g\, \rmd\psi_t^{\Gamma})^{n-1}
    \wedge \rmd t\\
                            &   &
    \mbox{}\wedge \bigl( f'\beta+g'\psi_t^{\Gamma}+
    g(\partial\psi_t^{\Gamma}/\partial t)\bigr)\\
                            & = &
    n\psi_t^{\Gamma}\wedge \rmd t\wedge
    \bigl((f'g-fg')\beta+g^2(\partial\psi_t^{\Gamma}/\partial t)\bigr)\\
                            &   &
    \mbox{}\wedge (f\, \rmd\beta+g\, \rmd\psi_t^{\Gamma})^{n-1}.
\end{eqnarray*}
Notice that $\partial\psi_t^{\Gamma}/\partial t$ is a horizontal $1$-form,
so the term where we wedge this with $\beta$ rather than
$\psi_t^{\Gamma}$ from the first factor yields a horizontal
$(2n+1)$-form, i.e.\ zero.

Near $t=-1$ we have $g\equiv 1$ and $\psi_t^{\Gamma}\equiv\psi_+^{\Gamma}$,
hence $\rmd\psi_t^{\Gamma}\equiv\omega_+^{\Gamma}$. With condition~(ii)
from Proposition~\ref{prop:decomposition} this implies
\[ \alpha\wedge (\rmd\alpha)^n=
nf'\psi_+^{\Gamma}\wedge \rmd t\wedge\beta\wedge
(f\, \rmd\beta+\omega_+^{\Gamma})^{n-1}>0.\]

Near $t=1$ we have $g\equiv -1$ and $\psi_t^{\Gamma}\equiv\psi_-^{\Gamma}$,
hence $\rmd\psi_t^{\Gamma}\equiv\omega_-^{\Gamma}$. Recall that
$\omega_-^{\Gamma}$ was defined as the restriction of $-\omega_-$,
so we get
\[ \alpha\wedge (\rmd\alpha)^n=
-nf'\psi_-^{\Gamma}\wedge \rmd t\wedge\beta\wedge
(f\, \rmd\beta-\omega_-^{\Gamma})^{n-1}>0.\]

Finally, over the region where $g'\neq 0$, we have $f\gg 1$
and $f'g-fg'\gg 1$.
It follows that the positive summand
\[ nf^{n-1}(f'g-fg')\psi_t^{\Gamma}\wedge \rmd t\wedge\beta\wedge
(\rmd\beta)^{n-1}\]
in the expression for $\alpha\wedge (\rmd\alpha)^n$ will
dominate all other summands.

The dividing set of the contact structure $\ker\alpha$
coincides with the zero set $\{ 0\}\times\Gamma$ of~$g$.
This completes the proof of the theorem.
\end{proof}
\section{Examples}
\subsection{From trivial to nontrivial bundles}
Here is a simple corollary of our main theorem.

\begin{cor}
\label{cor:bundles}
Let $B$ be a closed, oriented manifold of dimension~$2n$.
If the trivial $S^1$-bundle over $B$ admits an $S^1$-invariant
contact structure with dividing set~$\Gamma$, then so do all $S^1$-bundles
over~$B$.
\end{cor}

\begin{proof}
If the trivial $S^1$-bundle over $B$ admits an $S^1$-invariant contact
structure with dividing set~$\Gamma$, then by
Proposition~\ref{prop:decomposition} the base $B$ has a splitting
$B=B_+\cup_{\Gamma}B_-$ with a contact structure $\ker\beta$
on $\Gamma$ and exact symplectic forms $\rmd\lambda_{\pm}$
on $\pm B_{\pm}$ such that $(\pm B_{\pm},\rmd\lambda_{\pm})$
are weak fillings of $(\Gamma,\ker\beta)$.

Given an $S^1$-bundle over $B$ of Euler class $e\in H^2_{\dR}(B)$,
choose $2$-forms $\sigma_{\pm}$ on $B_{\pm}$ with
$\mp [\sigma_{\pm}/2\pi]=e|_{B_{\pm}}$. For $K\in\R^+$
sufficiently large, the $2$-forms $\omega_{\pm}:=
\sigma_{\pm}+K\, \rmd\lambda_{\pm}$ are symplectic forms satisfying
conditions (i) and (ii) of Proposition~\ref{prop:decomposition}.
Then the result follows from Theorem~\ref{thm:existence}.
\end{proof}

Using a result of Baykur~\cite{bayk06} on decompositions
of $4$-manifolds, it was shown in \cite[Corollary~2]{gest10} that the trivial
$S^1$-bundle over any closed, oriented $4$-manifold admits
an $S^1$-invariant contact structure. So the next corollary
is immediate. We give a direct proof illuminating the
role of Baykur's result.

\begin{cor}
\label{cor:five}
Any $S^1$-bundle over any closed, oriented $4$-manifold
admits an $S^1$-invariant contact structure.
\end{cor}

\begin{proof}
According to \cite[Theorem~5.2]{bayk06}, any closed, oriented
$4$-manifold $B$ admits a splitting $B=B_+\cup_{\Gamma} B_-$
with exact symplectic forms $\rmd\lambda_{\pm}$ on $\pm B_{\pm}$
that provide a filling of one and the same contact structure on~$\Gamma$.
In fact, the $(\pm B_{\pm},\rmd\lambda_{\pm})$ can be taken to be
Stein fillings.

Then the argument concludes as in the proof of the preceding
corollary.
\end{proof}

In \cite{gest10} it was also shown that $\CP^2\times S^1$
admits a contact structure in every homotopy class of almost contact
structures (i.e.\ reduction of the structure group to $\mathrm{U}(2)
\times 1$). The same argument as in the proof of Corollary~\ref{cor:bundles}
allows us to extend this to nontrivial bundles: on any given
$S^1$-bundle over $\CP^2$, any homotopy class of
$S^1$-invariant almost contact structures contains a contact structure.
\subsection{It's all in the twist}
\label{subsection:twist}
With the help of an example we want to illustrate
that the topological type of the total space $M$ of an $S^1$-bundle
with an invariant contact structure is not, in general,
determined by the symplectic splitting of the
base~$B$. This happens whenever there are nontrivial cohomology
classes in $H^2_{\dR}(B)$ that restrict to zero on the pieces
$B_{\pm}$ of the splitting, or in the presence of torsion
in $H^2(B;\Z)$.

We take $B=T^2\times S^2$ and split it into $B_{\pm}=T^2\times D^2_{\pm}$.
Write $\theta_1,\theta_2$ for circle coordinates on~$T^2$,
and $x,y$ (resp.\ $r,\varphi$) for Cartesian (resp.\ polar)
coordinates on $\pm D^2_{\pm}$. Then the gluing $D^2_+\cup_{S^1} D^2_-$
is given by $\varphi\mapsto\varphi$.

On $\pm B_{\pm}$ we have the symplectic form
$\rmd x\wedge\rmd\theta_1-\rmd y\wedge\rmd\theta_2$. This admits
the Liouville vector field $x\partial_x+y\partial_y$, and hence
gives a strong filling of $\Gamma=T^3=\partial (\pm B_{\pm})$
with contact structure $\ker(\cos\varphi\,\rmd\theta_1-\sin\varphi\,
\rmd\theta_2)$.

For any $k\in\Z$, the two trivial bundles $B_{\pm}\times S^1_{\theta}$
can be glued using the map
\[ g_k\co (\theta_1,\theta_2,\varphi,\theta)\longmapsto
(\theta_1,\theta_2,\varphi,\theta-k\varphi)\]
on the boundary. The Euler class $e\in H^2_{\dR}(B)$ of the
resulting $S^1$-bundle over $B$ satisfies
$\langle e,[T^2\times *]\rangle=0$ and $\langle e,[*\times S^2]\rangle=k$.
The restriction of $e$ to $B_{\pm}$ is the zero class.

According to Theorem~\ref{thm:existence}, each of these bundles over $B$
carries an $S^1$-invariant contact structure that induces
the same decomposition of $B$ into two exact symplectic pieces
$\pm B_{\pm}$. We shall return to this example
in Section~\ref{section:Liouville}.
\subsection{Open books and the Bourgeois construction}
\label{subsection:Bourgeois}
An {\em open book decomposition\/} of a manifold
$N$ consists of a codimension~$2$ submanifold $B_N$, called the
{\em binding}, and a (smooth, locally
trivial) fibration $\frakp\co N\setminus B_N
\rightarrow S^1$. The closures of the fibres $\frakp^{-1}(\varphi )$,
$\varphi\in S^1$, are called the {\em pages}. Moreover, it is
required that the binding $B_N$ have a trivial tubular neighbourhood
$B_N\times D^2$ in which $\frakp$ is given by the angular coordinate in
the $D^2\!$-factor.

If $N$ and $B_N$ are oriented, we orient the pages $p^{-1}(\varphi )$
consistently with their boundary~$B_N$. This is the same as saying that
$\partial_{\varphi}$ together with the orientation of the page gives
the orientation of~$N$, cf.~\cite[p.~154]{geig08}.

Following Giroux~\cite{giro02}, we say that a contact structure
$\xi=\ker\alpha$ on $N$ defined by a positive contact form $\alpha$ is
{\em supported\/} by the open book decomposition $(B_N,\frakp )$ if
\begin{enumerate}
\item[(i)] the $2$-form $\rmd\alpha$ induces a positive symplectic form on
each fibre of~$\frakp$, and
\item[(ii)] the $1$-form $\alpha$ induces a positive contact form on~$B_N$.
\end{enumerate}

Giroux has shown that every contact structure on a closed manifold
is supported by an open book. This fact was used by
Bourgeois~\cite{bour02} to show that, starting from a contact
structure on a closed manifold $N$, one can produce a
$T^2$-invariant contact structure on the product of $N$ with a $2$-torus
$T^2$. In particular, all odd-dimensional tori admit a contact structure.

We now want to indicate briefly how the Bourgeois construction
can be interpreted in the framework of the present note.
Thus, let $N$ be a closed, connected manifold
of dimension $\nobreak{2n-1}\geq 3$ with a contact
structure $\xi =\ker\alpha$
supported by an open book decomposition $(B_N,\frakp)$.
Let $(r,\varphi )$ be polar coordinates on the
$D^2$-factor of a neighbourhood $B_N\times D^2$ of the binding~$B_N$,
such that $\frakp\co N\setminus B_N\rightarrow S^1$ is given by
$\varphi$ in that neighbourhood. We choose this neighbourhood so
small that $\alpha$ restricts to a contact form on the manifold
$B_N\times\{ z\}$ for any $z\in D^2$.

Choose a smooth function $\rho$ of the variable $r$ on $B_N\times D^2$
satisfying the requirements that
\begin{itemize}
\item $\rho(r)= r$ near $B_N\equiv B_N\times\{0\}$,
\item $\rho'(r)\geq 0$,
\item $\rho\equiv 1$ near $B_N\times\partial D^2$.
\end{itemize}
Extend $\rho$ to a smooth function on $N$ by setting it equal to $1$
outside $B_N\times D^2$. Then $\ttx:=\rho\cos\varphi$
and $\tty:=\rho\sin\varphi$ are smooth functions on $N$ that
coincide with the Cartesian coordinate functions $x,y$ on the
$D^2$-factor near $B_N\times\{0\}\subset
B_N\times D^2$. The identity
\[ \ttx\,\rmd \tty-\tty\,\rmd \ttx=\rho^2\,\rmd\varphi\]
holds on all of~$N$.

In \cite{bour02}, cf.~\cite[Theorem~7.3.6]{geig08},
it was shown that
\[ \alpha-\ttx\,\rmd\phi+\tty\,\rmd\theta\]
defines a $T^2$-invariant contact structure on $N\times T^2=
N\times S^1_{\phi}\times S^1_{\theta}$. Our
Proposition~\ref{prop:decomposition} then tells us that
the $1$-form
\[ \beta_0:= \alpha -\ttx\,\rmd\phi \]
is an $S^1$-invariant contact form on
\[ \Gamma:= \{\tty=0\}\times S^1_{\phi}=
 \bigl((-\frakp^{-1}(0))\cup_{B_N}\frakp^{-1}(\pi)\bigr)\times S^1_{\phi},\]
and the $2$-form
\[ \omega_{\pm}:=\pm\rmd\Bigl(\frac{\alpha}{\tty}
                 -\frac{\ttx}{\tty}\,\rmd\phi\Bigr)\]
is an $S^1$-invariant symplectic form on $\Int(\pm B_{\pm})$, where
\[ B_+:=\frakp^{-1}([0,\pi])\times S^1_{\phi},\;\;\;
B_-:=\frakp^{-1}([\pi,2\pi])\times S^1_{\phi},\]
i.e.\ $B_{\pm}=\{\pm \tty\geq 0\}\times S^1_{\phi}$.
Conversely, one can check these properties directly and thus
derive Bourgeois's result.

Notice that the $S^1$-invariant contact structure on $\Gamma$
corresponds to the splitting of $\{\tty=0\}$ along
the hypersurface $\{\ttx=0=\tty\}$. The splitting of
$N\times S^1_{\phi}$ into two exact symplectic pieces $\pm B_{\pm}$ comes
from taking the two `halves' of the open book on $N$, crossed
with~$S^1_{\phi}$. This allows one to give simple explicit
symplectic splittings of manifolds of the form $N\times S^1$.

With Corollary~\ref{cor:bundles}
we obtain the following generalisation of Bourgeois's result.

\begin{cor}
Let $N$ be a closed manifold admitting a contact structure.
Then any principal $S^1$-bundle over $N\times S^1$ admits an
$S^1$-invariant contact structure. \qed 
\end{cor}
\section{Ideal Liouville domains}
\label{section:Liouville}
The manuscript of~\cite{mnw} became available only after the first
version of the present note had been completed. In this
section, which is based on correspondence with Patrick Massot, we wish to
explain how our construction can be phrased in the more
sophisticated language of that paper.

The following definition, taken from \cite[Section~4.2]{mnw},
is due to Giroux.

\begin{defn}[Giroux]
An \emph{ideal Liouville domain} is a triple $(\Sigma,\omega_{\Int},\eta)$
consisting of a compact oriented $2n$-manifold $\Sigma$
with boundary, a symplectic form $\omega_{\Int}$ on the
interior $\Int (\Sigma)$, and a contact structure $\eta$
on the boundary $\partial\Sigma$, such that there is an auxiliary
$1$-form $\lambda$ on $\Int (\Sigma)$ with the following
properties:
\begin{enumerate}
\item[(i)] $\rmd\lambda=\omega_{\Int}$;
\item[(ii)] for some (and hence any) smooth function
$u\co\Sigma\rightarrow\R_0^+$ with $\partial\Sigma$ as its regular zero set,
the $1$-form $\beta:=u\lambda$ on $\Int (\Sigma)$ extends to $\partial\Sigma$
as a contact form for~$\eta$.
\end{enumerate}
\end{defn}

Given an ideal Liouville domain $(\Sigma,\omega_{\Int},\eta)$, it
is easy to check that the $1$-form $u\lambda+u\,\rmd\theta$
defines an $S^1$-invariant contact structure $\xi$ on $\Sigma\times S^1$.
In~\cite{mnw}, the pair $(\Sigma\times S^1,\xi)$ is called
the \emph{Giroux domain} associated with $(\Sigma,\omega_{\Int},\eta)$.

We now adapt these definitions to nontrivial $S^1$-bundles.

\begin{defn}
Let $\Sigma$ and $\eta$ be as above, $\omega_{\Int}$
a symplectic form on $\Int (\Sigma)$,
and $c$ a cohomology class in $H^2(\Sigma;\Z)$.
The tuple $(\Sigma,\omega_{\Int}, \eta,c)$ is called an
\emph{ideal Liouville domain} if for some
(and hence any) closed $2$-form $\omega$ on $\Sigma$ with
$-[\omega/2\pi]=c\otimes\R\in H^2_{\dR}(\Sigma)$ there exists
a $1$-form $\lambda$ on $\Int(\Sigma)$ such that
$\rmd\lambda=\omega_{\Int}-\omega$ and condition (ii) holds as before.
\end{defn}

The choice of $\omega$ is indeed irrelevant.
If $\omega$ is replaced by $\omega'$ with $[\omega']=[\omega]$,
then $\omega'=\omega+\rmd\mu$ with $\mu$ a $1$-form on~$\Sigma$.
So we simply need to replace $\lambda$ by $\lambda'=\lambda-\mu$.
Since $\mu$ is defined on $\Sigma$ (including the boundary),
the extension of $u\lambda'$ to $\partial\Sigma$
coincides with that of $u\lambda$. Observe that an
ideal Liouville domain $(\Sigma,\omega_{\Int},\eta,0)$
is an ideal Liouville domain in the sense of Giroux.

Now let $\pi\co M\rightarrow\Sigma$ be the principal $S^1$-bundle over
$\Sigma$ of (integral) Euler class~$c$. Choose a connection $1$-form $\psi$
with curvature form $\omega$ on this bundle. As before,
it is easy to check that
the $1$-form $u\psi+u\gamma$ defines a contact structure $\xi$ on~$M$.
Notice that $\xi$ intersects $\pi^{-1}(\partial\Sigma)$
transversely; the intersection  $\xi\cap T(\pi^{-1}(\partial\Sigma))$ is
the tangent hyperplane field given by the kernel of
(the lift of)~$\beta$. We call $(M,\xi)$ the \emph{contactisation}
of $(\Sigma,\omega_{\Int},\eta,c)$.

\begin{defn}
Let $B$ be a closed, oriented manifold of dimension $2n$.
An \emph{ideal Liouville splitting of class $c\in H^2(B;\Z)$}
is a decomposition $B=B_+\cup_{\Gamma}B_-$ along a two-sided
(but not necessarily connected) hypersurface $\Gamma$, oriented
as the boundary of~$B_+$, together
with a contact structure $\eta$ on $\Gamma$ and symplectic forms
$\omega_{\pm}$ on $\Int(\pm B_{\pm})$, such that
\[ (\pm B_{\pm},\omega_{\pm},\eta,\pm c|_{B_{\pm}})\]
are ideal Liouville domains.
\end{defn}

In this terminology, Proposition~\ref{prop:decomposition}
can be read as saying that an $S^1$-invariant contact structure
$\xi=\ker(\beta+u\psi)$ on the principal
$S^1$-bundle $M\rightarrow B$ defined by $c\in H^2(B;\Z)$
leads to a Liouville splitting of $B$ of class~$c$.
It will become apparent presently why we need
to fix an \emph{integral} class $c$ in our notion
of ideal Liouville splitting.

Conversely, assume that
the base $B$ of the $S^1$-bundle $M\rightarrow B$ of Euler class
$c\in H^2(B;\Z)$ admits an ideal Liouville splitting of
class~$c$. The contactisations $(M_{\pm},\xi_{\pm})$ of $B_{\pm}$
induce the same non-singular and $S^1$-invariant
characteristic distribution $\xi_{\pm}\cap T(\pi^{-1}(\Gamma))$.
Therefore, any equivariant diffeomorphism of $\pi^{-1}(\Gamma)$,
in particular the one necessary to recover~$M$,
can be used to obtain a contact manifold by gluing these two
contactisations. This is true thanks to the
following proposition, which is folklore.

\begin{prop}
Let $H$ be a closed hypersurface in a contact manifold $(M,\xi=\ker\alpha)$.
Then the germ of $\xi$ near $H$ is determined by the
$1$-form $\alpha|_{TH}$. In particular, if $\xi|_H$
is transverse to~$H$, then the germ of $\xi$ near $H$ is
determined by the codimension~$1$ distribution $TH\cap\xi$ on~$H$.
\end{prop}

\begin{proof}
By passing to a double cover, if necessary, we may assume that
$H$ is orientable. Let $\xi_0=\ker\alpha_0$ and $\xi_1=\ker\alpha_1$
be two contact structures
near $H$ with $\alpha_0|_{TH}=\alpha_1|_{TH}$.
Identify a neighbourhood of $H$
with $H\times\R$. Then we can write
$\alpha_i=\beta_i^r+u_i^r\,\rmd r$, $i=0,1$,
where $\beta_i^r$, $r\in\R$, is a smooth family
of $1$-forms on~$H$, and $u_i^r\co H\rightarrow\R$
a smooth family of functions. By assumption we have
$\beta_0^0=\beta_1^0$ and hence $\rmd\beta_0^0=\rmd\beta_1^0$.
The contact condition for the
$\alpha_i$ looks as follows, where we drop the subscript $i$
for the moment (and $\dim M=2n+1$):
\[ \bigl(-n\beta^r\wedge (\partial\beta^r/\partial r)+
n\beta^r\wedge\rmd u^r+u^r\,\rmd\beta^r\bigr)\wedge (\rmd\beta^r)^{n-1}
\wedge\rmd r>0.\]
This expression is linear in $\partial\beta^r/\partial r$ and~$u^r$.
It follows that on a neighbourhood of $H$
the convex linear interpolation $\alpha_t=(1-t)\alpha_0+t\alpha_1$
is a contact form for all $t\in [0,1]$.

Now apply the Moser trick to this family, cf.~\cite[Chapter~2]{geig08}.
One finds that $\xi_0$ is isotopic to $\xi_1$ via an isotopy
$\psi_t$ given as the flow of a vector field $X_t\in\ker\alpha_t$
determined by the equation
\[ \dot{\alpha}_t+i_{X_t}\rmd\alpha_t=\mu_t\alpha_t,\]
where the function $\mu_t$ is given by evaluating $\dot{\alpha}_t$
on the Reeb vector field of~$\alpha_t$.

On $TH$ we have $\dot{\alpha}_t=0$. So for $\bfv\in TH\cap\xi_i$
we have $\rmd\alpha_t(X_t,\bfv)=0$. This implies that
$X_t|_H\in TH\cap\xi_i$, i.e.\ the flow of $X_t$ preserves both $H$
and the characteristic distribution. Since $H$ is closed,
$X_t$ integrates to a flow $\psi_t$, $t\in [0,1]$, near~$H$.
\end{proof}

In the present situation all forms in question are $S^1$-invariant.
Then so will be the isotopy~$\psi_t$. This makes
the gluing $S^1$-equivariant.

Beware that, as in the example in Section~\ref{subsection:twist}, there
may be a choice of gluing that can affect the global topology.
Indeed, along two boundary components with the same
characteristic distribution one can glue using any 
$S^1$-equivariant diffeomorphism that preserves this distribution.

Let us consider that example in a little more detail.
The manifold $T^2\times D^2$ becomes an ideal Liouville
domain (in the narrow sense, i.e.\ with $c=0$) by setting
\[ \omega_{\Int}=\rmd\Bigl(\frac{x}{1-r^2}\,\rmd\theta_1-
\frac{y}{1-r^2}\,\rmd\theta_2\Bigr)\]
and $\eta=\ker(\cos\varphi\,\rmd\theta_1-\sin\varphi\,\rmd\theta_2)$.
The gluing maps $g_k$ preserve the characteristic distribution
induced on $T^2\times S^1\times S^1$ by the contactisation,
and hence can be used to glue two such contactisations.

Thus, we see that an ideal Liouville splitting
without a prescribed class $c\in H^2(B;\Z)$ may not lead to
a unique $(M,\xi)$. If there is torsion in $H^2(B_{\pm};\Z)$,
the class $e=c\otimes\R$ does not determine the topology
of the contactisations; if there is a class in $H^2(B;\Z)$ that
restricts to zero on $B_{\pm}$, the topology of the
manifold obtained by gluing the two
contactisations may not be determined by $c|_{B_{\pm}}$.

The following theorem generalises the equivariant classification
result of Lutz~\cite{lutz77}.

\begin{thm}
Let $\pi\co M\rightarrow B$ be the $S^1$-bundle of
Euler class $c\in H^2(B;\Z)$. Two invariant contact
structures on $M$ are equivariantly diffeomorphic if and only
if they induce diffeomorphic Liouville splittings of $B$ of
class~$c$, where it is understood that the diffeomorphism
of $B$ preserves the class~$c$.
\end{thm}

\begin{proof}
We consider the case $\Gamma\neq\emptyset$. All arguments go through
for the case $\Gamma=\emptyset$, i.e.\ $B=B_+$ or $B=B_-$,
with the obvious modifications.

For the `only if' direction we first need to check that
the Liouville splitting induced by an invariant contact structure does
not depend on choices. For notational convenience we only consider~$B_+$.

Given a connection $1$-form $\psi$ on $M$ with curvature $2$-form
$\omega$, and an invariant contact form $\alpha
=\beta+u\psi$, we get the ideal Liouville domain
$(B_+,\omega_+,\eta, c|_{B_+})$, where
\[ \omega_+=\bigl(\rmd (\beta/u)+\omega\bigr)|_{\Int(B_+)}\]
and $\eta=\ker (\beta|_{T\Gamma})$. If the connection $1$-form
is replaced by $\psi'=\psi+\gamma$, with $\gamma$ a
$1$-form lifted from~$B$, then $\beta$ needs to be replaced by
$\beta'=\beta-u\gamma$, and $\omega$ is replaced by
$\omega'=\omega+\rmd\gamma$. The symplectic form $\omega_+$ remains
unchanged.
If the contact form $\alpha$ is replaced by $f\alpha$ for some
invariant function $f\co M\rightarrow\R^+$, then $\beta'=f\beta$
and $u'=fu$, so again $\omega_+$ does not change.
Also, these choices do not affect the contact structure $\eta$ on~$\Gamma$.

An equivariant diffeomorphism of invariant contact structures on $M$
clearly induces a diffeomorphism of the induced ideal Liouville
splittings. Since the $S^1$-bundle is determined by the
class $c\in H^2(B;\Z)$, the induced diffeomorphism of $B$
preserves this class.

Conversely, i.e.\ for the `if' direction,
suppose that $\alpha$ and $\alpha'$ are
invariant contact forms on $M$ inducing
Liouville splittings (of class~$c$) of $B$ that are diffeomorphic via
some diffeomorphism $\phi\co B\rightarrow B$. Write
$\phi^*M$ for the total space of the $S^1$-bundle over $B$ obtained by
pulling back the bundle $M\rightarrow B$. There is an
equivariant diffeomorphism $\tilde{\phi}\co\phi^*M\rightarrow M$
covering~$\phi$, and $\tilde{\phi}^*\alpha'$ induces the
\emph{same} ideal Liouville splitting as~$\alpha$.

Since $\phi$ preserves the class $c$ in integral cohomology,
the bundle $\phi^*M\rightarrow B$ is bundle isomorphic
to $M\rightarrow B$, i.e.\ there is an equivariant
diffeomorphism $M\rightarrow \phi^*M$ over the identity.
So we may regard $\phi^*\alpha'$ as an invariant contact
form on $M$, defining the same ideal Liouville
splitting as~$\alpha$.

Write $\alpha=\beta+u\psi$. Then without loss of generality
we can write $\phi^*\alpha'=\beta'+u\psi$.
Both contact forms induce the same symplectic form $\omega_+$
on $\Int (B_+)$, hence
\[ \rmd(\beta/u)|_{\Int(B_+)}=\rmd(\beta'/u)|_{\Int(B_+)}.\]
Moreover,
$\beta$ and $\beta'$ induce the same contact structure $\eta$
on~$\Gamma$, i.e.\
\[ \ker(\beta|_{T\Gamma})=\ker(\beta'|_{T\Gamma}).\]
Recall from Section~\ref{section:decomposition} that
$\alpha\wedge(\rmd\alpha)^n= \psi\wedge\Omega$, with
\[ \Omega|_{\Int(B_+)}=u^{n+1}\bigl(\rmd(\beta/u)+\omega\bigr)^n\]
and
\[ \Omega|_{\Gamma}=-n\,\rmd u\wedge\beta\wedge(\rmd\beta)^{n-1}.\]
It follows that $(1-t)\beta+t\beta'+u\psi$ is
a contact form for all $t\in [0,1]$. Gray stability then yields
an $S^1$-equivariant contact isotopy between $\ker\alpha$
and $\ker\phi^*\alpha'$.
\end{proof}

The proof of Corollary~\ref{cor:bundles} can likewise be
phrased in terms of ideal Liouville domains.
\begin{ack}
We are grateful to Patrick Massot for explaining to us how
to interpret our result in the language of ideal Liouville
domains. We thank the referee for constructive comments on an earlier
version of this paper.
\end{ack}

\end{document}